\documentclass[bibalpha]{amsart}
\usepackage{verbatim}
\usepackage{amssymb}
\usepackage{extarrows}
\usepackage{url}
\usepackage{fullpage}

\title{When does NIP transfer from fields to henselian expansions?}

\author{Franziska Jahnke} 
\address{Institut f\"ur Mathematische Logik\\Einsteinstr. 62\\48149 M\"unster, 
Germany}
\email{franziska.jahnke@uni-muenster.de}

\thanks{}

\newtheorem{Thm}{Theorem}[section]
\newtheorem{Ex}[Thm]{Example}

\newtheorem*{Def}{Definition}
\newtheorem*{Thm*}{Theorem}
\newtheorem{Cor}[Thm]{Corollary}
\newtheorem{Fact}[Thm]{Fact}
\newtheorem{Prop}[Thm]{Proposition}

\newtheorem{Q}[Thm]{Question}
\newtheorem{Qs}[Thm]{Questions}

\begin{document}
\begin{abstract} 
Let $K$ be an NIP field and let $v$ be a henselian valuation on $K$. We ask whether $(K,v)$ is NIP as a valued field. 
By a result of Shelah, we know that if $v$ is externally definable, then $(K,v)$ is NIP. 
Using the definability of the canonical $p$-henselian valuation, we show that whenever the residue field of $v$ is not separably closed, then $v$ is externally 
definable. In the case of separably closed residue field, we show that $(K,v)$ is NIP as a pure valued field.
\end{abstract}

\maketitle
\section{Introduction and Motivation}
There are many open questions connecting NIP and henselianity, most prominently
\begin{Q}
\begin{enumerate}
\item Is any valued NIP field $(K,v)$ henselian?
\item Let $K$ be an NIP field, neither separably closed nor real closed. Does $K$ admit a definable non-trivial
henselian valuation?
\end{enumerate}
\end{Q}
Both of these questions have been recently answered positively in the special case where `NIP' is
replaced with `dp-minimal' (cf. Johnson's results in \cite{Joh15}). Moreover, Johnson also showed that question (1) can be answered affirmatively when the characteristic of $K$ is positive and showed a positive answer to
question (2) when `NIP' is replaced by `dp-finite of positive characteristic' (\cite{Joh19}).

The question discussed here is the following:
\begin{Q} \label{mainq}
Let $K$ be an NIP field (in an expansion of the language of rings)
and $v$ a henselian valuation on $K$. Is $(K,v)$ NIP?
\end{Q}
Note that this question neither implies nor is implied by any of the above questions, it does however follow along the same lines
aiming to find out how close the bond between NIP and henselianity really is.
 
The first aim of this article is to show that the answer to Question 
\ref{mainq} is `yes' if $Kv$ is not separably closed:
{
\renewcommand{\theThm}{A}
\begin{Thm} \label{A}
Let $(K,v)$ be henselian and such that $Kv$ is not separably closed.
Then $v$ is definable in the Shelah expansion $K^\mathrm{Sh}$.
\end{Thm}
}
See section \ref{Sh} for the definition of $K^\mathrm{Sh}$.
The theorem follows immediately from combining Propositions \ref{notreal} and \ref{real}.  
If $v$ is definable in $K^\mathrm{Sh}$, then one can add a symbol for the valuation ring $\mathcal{O}$ to any language 
$\mathcal{L}$ extending $\mathcal{L}_\textrm{ring}$ and obtain that if $K$ is NIP as an $\mathcal{L}$-structure, then
$(K,v)$ is NIP as an $\mathcal{L} \cup \{\mathcal{O}\}$-structure.
Theorem \ref{A} is proven using the definability of the canonical $p$-henselian valuation. We make a case distinction between
when $Kv$ is neither separably closed nor real closed 
(Proposition \ref{notreal}) and when $Kv$ is real closed (Proposition
\ref{real}).

%\smallskip
On the other hand, 
if $Kv$ is separably closed, then we cannot hope for a result in the same generality: 
it is well-known that 
any
algebraically closed valued field is NIP in $\mathcal{L}_\textrm{ring} \cup \{\mathcal{O}\}$, however,
any algebraically closed field with two independent valuations has IP (\cite[Theorem 6.1]{Joh13}, see also Example \ref{extra}).
In this case, we can still consider the question in the language of rings: given an NIP field $K$ and a henselian valuation $v$
on $K$, is $(K,v)$ NIP in $\mathcal{L}_\textrm{ring}\cup\{\mathcal{O}\}$? 
The answer to this is again positive:
{\renewcommand{\theThm}{B}
\begin{Thm} \label{B}
Let $K$ be NIP, $v$ henselian on $K$.
Then $(K,v)$ is NIP as
a pure valued field. 
\end{Thm}}
Theorem \ref{B} is proven as Theorem \ref{main2} in section \ref{scf}. The proof of the theorem uses two NIP transfer theorems 
recently proven in \cite{JS} and \cite{AJ19b}. A transfer theorem gives criteria under which dependence of the residue field
implies dependence of the (pure) valued field. Delon proved a transfer theorem for henselian valued fields
of equicharacteristic $0$ (see \cite{Del} or {\cite[Theorem A.15]{Simon:book}}), and
B\'elair proved 
a version for equicharacteristic Kaplansky fields which are algebraically maximal (see \cite{Bel}).
The transfer theorems proven in \cite[Theorem 3.3]{JS} and \cite[Proposition 4.1]{AJ19b} generalizes these results to separably algebraically maximal 
Kaplansky fields, in particular, they also works in mixed characteristic. See section \ref{scf} for definitions and more details.
Combining these transfer theorems with an idea of Scanlon and some standard trickery concerning definable valuations 
yields that under the assumptions of Theorem \ref{B}, we can find a decompostion of $v = \bar{v} \circ w$
into two NIP valuations.
The
question whether the composition of two henselian NIP valuations is again NIP seems to be open.
For the case when the residue field of the coarser valuation is stably embedded, this follows from \cite[Proposition 2.5]{JS}.
Using that the residue field in separably algebraically maximal Kaplansky fields and unramified henselian fields is
stably embedded, 
this allows us to prove the second part of Theorem \ref{B}.

%\smallskip
The paper is organized as follows: In section \ref{nscf}, we first recall the necessary background concerning the Shelah expansion.
We then discuss the definition and definability of the canonical $p$-henselian valuation. In the final part, we use these two
ingredients to prove Theorem \ref{A}. In particular, we conclude that for any henselian NIP field the residue
field is NIP as a pure field. We also obtain as a consequence that if a field admits a non-trivial henselian valuation
and is NIP in some $\mathcal{L} \supseteq \mathcal{L}_\mathrm{ring}$, then there is \emph{some} 
non-trivial henselian valuation $v$ on $K$
such that $(K,v)$ is NIP in $\mathcal{L} \cup \{O\}$ (Corollary \ref{exist}).

In the third section, we treat the case of separably closed residue fields. We first recall an example which
shows that we have to restrict Question \ref{mainq} to the language of pure valued fields.
We then briefly review different ingredients, starting with the transfer theorem for separably algebraically maximal Kaplansky
fields.
After quoting a result by Delon that separably closed valued fields are NIP, 
we state and prove a Proposition by Scanlon (Proposition \ref{scan}) which implies
that on an NIP field, any valuation with non-perfect residue field is $\mathcal{L}_\mathrm{ring}$-definable. We then recall
some facts about stable embeddedness and the composition of NIP valuations.
In the final subsection, we prove Theorem \ref{B}.

Finally, in section \ref{ord}, we treat the simpler case of convex valuation rings on an ordered field $(K,v)$. As any
convex valuation ring is definable in $(K,<)^\mathrm{Sh}$, we conclude that if $(K,<)$ is an ordered NIP field in some
language $\mathcal{L} \supseteq \mathcal{L}_\mathrm{ring} \cup \{<\}$ and $v$ is a convex valuation on $K$,
then $(K,v)$ is NIP in $\mathcal{L} \cup \{\mathcal{O}\}$ (Corollary \ref{cord}).

%\smallskip
Throughout the paper, we use the following notation: 
for a valued field $(K,v)$, we write $vK$ for the value group, $Kv$ for the residue field
and $\mathcal{O}_v$ for the valuation ring of $v$.

\section{Non-separably closed residue fields} \label{nscf}
\subsection{Externally definable sets} \label{Sh}
Throughout the subsection, let $M$ be a structure in some language $\mathcal{L}$.
\begin{Def}
Let $N \succ M$ be an $|M|^+$-saturated elementary extension.
A subset $A \subseteq M$ is called \emph{externally definable} if it is of the form
$$\{a \in M^{|\bar{x}|}\,|\,N \models \varphi(a,b)\}$$
for some $\mathcal{L}$-formula $\varphi(\bar{x},\bar{y})$ and some $b \in N^{|\bar{y}|}$. 
\end{Def}
The notion of externally definable sets does not depend on the choice of $N$. See {\cite[Chapter 3]{Simon:book}}
for more details on externally definable sets.

\begin{Def}
The \emph{Shelah expansion} $M^\mathrm{Sh}$ is the expansion of $M$ by predicates for all externally definable sets.
\end{Def}

Note that the Shelah expansion behaves well when it comes to NIP:
\begin{Prop}[Shelah, {\cite[Corollary 3.14]{Simon:book}}]
If $M$ is NIP then so is $M^\mathrm{Sh}$. \label{Shelah}
\end{Prop}

The way the Shelah expansion is used in this paper is to show that any coarsening of a definable valuation on an NIP field
is an NIP valuation. Thus, the following example is crucial:
\begin{Ex} \label{mainex}
Let $(K,w)$ be a valued field and $v$ be a coarsening of $w$, i.e., a valuation
on $K$ with $\mathcal{O}_v \supseteq \mathcal{O}_w$. 
Then, there is a convex subgroup $\Delta \leq wK$
such that we have $vK \cong wK/\Delta$. As $\Delta$ is externally definable in the ordered abelian group $wK$,
the valuation ring $\mathcal{O}_v$ is definable in $(K,w)^\mathrm{Sh}$.
\end{Ex}

\subsection{$p$-henselian valuations}
Throughout this subsection, let $K$ be a field and $p$ a prime. We recall the main properties of the canonical $p$-henselian valuation
on $K$.
We define $K(p)$ to be the compositum of all Galois extensions of $K$ of 
$p$-power degree (in a fixed algebraic closure). Note that we have
\begin{itemize}
\item $K \neq K(p)$ iff $K$ admits a Galois extension of degree $p$ and
\item if $[K(p):K]<\infty$ then $K=K(p)$ or $p=2$ and $K(2)=K(\sqrt{-1})$ (see \cite[Theorem 4.3.5]{EP05}).
\end{itemize} 
A field $K$ which admits exactly one Galois extension of $2$-power degree is called \emph{Euclidean}. Any 
Euclidean field is uniquely ordered, the positive elements being exactly the squares (see \cite[Proposition 4.3.4 and Theorem 4.3.5]{EP05}). 
In particular, the ordering on a Euclidean field is $\mathcal{L}_\mathrm{ring}$-definable.

\begin{Def}
A valuation $v$ on a field $K$ is called
\emph{$p$-henselian} if $v$ extends uniquely to $K(p)$.
We call $K$ \emph{$p$-henselian} if $K$ 
admits a non-trivial
$p$-henselian valuation. 
\end{Def}

Every henselian valuation is $p$-henselian for all primes $p$.
Assume $K \neq K(p)$. Then, there is a canonical
$p$-henselian valuation on $K$:
We divide the class of 
$p$-henselian valuations on $K$
 into
two subclasses,
$$H^p_1(K) = \{v\; p\textrm{-henselian on } K \,| \,Kv \neq Kv(p)\}$$
and
$$H^p_2(K) = \{ v\; p\textrm{-henselian on } K \,|\, Kv = Kv(p) \}.$$ 
One can show that any valuation $v_2 \in H^p_2(K)$ 
is \emph{finer} than any $v_1 \in H^p_1(K)$, i.e. 
${\mathcal O}_{v_2} \subsetneq {\mathcal O}_{v_1}$,
and that any two valuations in $H^p_1(K)$ are comparable.
Furthermore, if $H^p_2(K)$ is non-empty, then there exists a unique coarsest
valuation
$v_K^p$ in $H^p_2(K)$; otherwise there exists a unique finest 
valuation $v_K^p \in H^p_1(K)$.
In either case, $v_K^p$ is called the \emph{canonical $p$-henselian valuation} (see \cite{Koe95} for more details).

The following properties of the canonical $p$-henselian valuation follow immediately from the definition:
\begin{itemize}
\item If $K$ is $p$-henselian then $v_K^p$ is non-trivial.
\item Any $p$-henselian valuation on $K$ is comparable to $v_K^p$.
\item If $v$ is a $p$-henselian valuation on $K$ with $Kv \neq Kv(p)$, then $v$ coarsens $v_K^p$.
\item If $p=2$ and $Kv_K^2$ is Euclidean, then there is a (unique) $2$-henselian valuation $v_K^{2*}$ such that
$v_K^{2*}$ is the coarsest $2$-henselian valuation with Euclidean residue field.
\end{itemize}

\begin{Thm}[{\cite[Corollary 3.3]{JK14}}] \label{JK14}
Let $p$ be a prime and 
consider the (elementary) class of fields
$$\mathcal{K} = \{ K \,|\, K \;p\textrm{-henselian, with }\zeta_p \in K
\textrm{ in case } \mathrm{char}(K)\neq p\}$$
There is a parameter-free $\mathcal{L}_\textrm{ring}$-formula $\psi_p(x)$
such that 
\begin{enumerate}
\item if $p \neq 2$ or $Kv_K^2$ is not Euclidean, then $\psi_p(x)$ defines 
the valuation ring of the canonical
$p$-henselian valuation $v_K^p$, and
\item if $p=2$ and $Kv_K^2$ is Euclidean, then $\psi_p(x)$ defines the valuation
ring of the
coarsest $2$-henselian valuation $v_K^{2*}$ such that
$Kv_K^{2*}$ is Euclidean. 
\end{enumerate}
\end{Thm}

\subsection{External definability of henselian valuations}
In this subsection, we apply the results from the previous two subsections to prove Theorem \ref{A} from the introduction.
\begin{Prop} \label{notreal}
Let $(K,v)$ be henselian such that $Kv$ is neither separably closed nor real closed.
Then $v$ is definable in $K^\mathrm{Sh}$.
\end{Prop}
\begin{proof}
Assume $Kv$ is neither separably closed nor real closed. For any finite separable extension $F$ of $K$, we use $u$ to denote the
(by henselianity unique) extension of $v$ to $F$. 
Choose any prime $p$ such that $Kv$ has a finite Galois extension $k$
of degree divisible by $p^2$. 
Consider a finite Galois extension $N \supseteq K$ such that $Nu=k$. 
Note that such an $N$ exists by \cite[Corollary 4.1.6]{EP05}.
Now, let $P$ be a $p$-Sylow of $\mathrm{Gal}(Nu/Kv)$. Recall that the canonical restriction map
$$\mathrm{res}: \mathrm{Gal}(N/K) \to \mathrm{Gal}(Nu/Kv)$$
is a surjective homomorphism (\cite[Lemma 5.2.6]{EP05}). 
Let $G \leq \mathrm{Gal}(N/K)$ be the preimage of $P$ under this map, and let
$L:=\mathrm{Fix}(G)$ be the intermediate field fixed by $G$. In particular, $L$ is a finite separable extension of $K$.
By construction, the extension $Lu \subseteq Nu$ is a Galois extension of degree $p^n$ for some $n \geq 2$, 
in particular, we have $Lu\neq Lu(p)$.  

Hence, we have constructed some finite separable extension $L$ of $K$ with $Lu\neq Lu(p)$.
Moreover, we may assume that $L$ contains a
primitive $p$th root of unity in case $p\neq 2$ and $\mathrm{char}(K)\neq p$:
The field $L':=L(\zeta_p)$ is again a finite separable extension of $K$ and its residue field is a finite extension of $Lu$. 
Thus, by \cite[Theorem 4.3.5]{EP05}, we get $L'u \neq L'u(p)$. Similarly,  
in case $p = 2$ and $\mathrm{char}(K)=0$, we may assume that $L$ contains a square root of $-1$: By construction,
$Lu$ has a Galois extension of degree $p^n$ for some $n \geq 2$. Consider $L':=L(\sqrt{-1})$, 
then $L'u$ is not $2$-closed and not orderable. 
In this case, no $2$-henselian
valuation on $L'$ has Euclidean residue field (see \cite[Lemma 4.3.6]{EP05}).

Finally, $v_L^p$
is definable on $L$ by a parameter-free $\mathcal{L}$-formula $\varphi_p(x)$.
It follows from the defining properties of $v_L^p$ that $\mathcal{O}_{v_L^p} \subseteq \mathcal{O}_{u}$ holds.
As $L/K$ is finite, $L$ is interpretable in $K$.
Hence, $\mathcal{O}_w:=\mathcal{O}_{v_L^p} \cap K$ is an $\mathcal{L}_\textrm{ring}$-definable valuation ring
of $K$ with $\mathcal{O}_w \subseteq \mathcal{O}_v$. By Example 
\ref{mainex}, $v$ is definable in $K^\mathrm{Sh}$. 
\end{proof}

\begin{Prop} \label{real}
Let $(K,v)$ be henselian such that $Kv$ is real closed.
Then $v$ is definable in $K^\mathrm{Sh}$.
\end{Prop}
\begin{proof}
Assume that $(K,v)$ is henselian and $Kv$ is real closed. Then $K$ is orderable. 
%We may assume that $v$ is non-trivial since
%otherwise $v$ is obviously definable in $K^{Sh}$. In this case, $K$ us 
We first reduce to the case that $K$ is Euclidean: Note that $v$ is a $2$-henselian valuation with
Euclidean residue field. Let $v_K^{2*}$ be the coarsest $2$-henselian valuation on $K$ with Euclidean residue field, which
 is $\emptyset$-definable on $K$ in $\mathcal{L}_\mathrm{ring}$ by Theorem \ref{JK14}. 
Now, if the induced valuation $\overline{v}$ on $Kv_K^{2*}$
is definable in  $(Kv_K^{2*})^\mathrm{Sh}$, 
then the valuation ring of $v$, which
is the composition of $v_K^{2*}$ and $\overline{v}$, is 
also definable in $K^\mathrm{Sh}$.

Thus, we may assume that $K$ is Euclidean. In this case, $K$ is uniquely ordered
%In case $vK$ is not $2$-divisible, some refinement $u$
%of $v$ is $\emptyset$-definable in $\mathcal{L}_\textrm{ring}$ such that all henselian valuations on $Ku$
%have $2$-divisible value group (see \cite[Theorem 3.1]{JK14}). Thus, $Ku$ is euclidean and hence uniquely orderable.
and the ordering on $K$ is $\mathcal{L}_\textrm{ring}$-definable. 
Let $\mathcal{O}_{{w}} \subseteq K$ be the convex hull of $\mathbb{Z}$ in $K$. Then, $\mathcal{O}_w$ is
definable in $K^\mathrm{Sh}$. 
By \cite[Theorem 4.3.7]{EP05}, $(K,w)$ is a $2$-henselian valuation 
ring on $K$ with Euclidean residue field. As $w$ has no proper refinements,
$w$ is the canonical $2$-henselian valuation on $K$.
Thus, we get $\mathcal{O}_w\subseteq \mathcal{O}_v$ and hence $\mathcal{O}_v$ is also definable in $K^\mathrm{Sh}$ by Example \ref{mainex}.
\end{proof}

Note that combining Propositions \ref{notreal} and \ref{real} immediately yields Theorem \ref{A} from the
introduction. Applying Proposition \ref{Shelah}, we conclude:

\begin{Cor} Let $K$ be a field and $v$ a henselian valuation on $K$. Assume that $\mathrm{Th}(K)$ is NIP in some language 
$\mathcal{L}\supseteq \mathcal{L}_\mathrm{ring}$. \label{nonsep}
If $Kv$ is not separably closed, then $(K,v)$ is NIP in the language $\mathcal{L}\cup \{\mathcal{O}\}$.
\end{Cor}

As separably closed fields are always NIP in $\mathcal{L}_\mathrm{ring}$, we note that the residue field
of a henselian valuation on an NIP field is NIP as a pure field.
\begin{Cor} \label{resnip}
Let $K$ be a field and $v$ henselian on $K$. Assume that $\mathrm{Th}(K)$ is NIP in some language 
$\mathcal{L}\supseteq \mathcal{L}_\mathrm{ring}$. Then $Kv$ is NIP as a pure field.
\end{Cor}

Recall that a field $K$ is called \emph{henselian} if it admits some non-trivial henselian valuation.
\begin{Cor} \label{exist}
Let $K$ be a henselian field such that $\mathrm{Th}(K)$ is NIP in some language 
$\mathcal{L}\supseteq \mathcal{L}_\mathrm{ring}$. Assume that $K$ is neither separably closed
nor real closed. Then $K$ admits some non-trivial 
externally definable henselian valuation $v$.
In particular, $(K,v)$ is NIP in the language $\mathcal{L}\cup \{\mathcal{O}_v\}$.
\end{Cor}
\begin{proof}
If $K$ admits some non-trivial henselian valuation $v$ such that $Kv$ is not separably closed, the result follows immediately
by Propositions \ref{notreal} and \ref{real}. Otherwise, $K$ admits a non-trivial $\mathcal{L}_\mathrm{ring}$-definable
henselian valuation by \cite[Theorem 3.8]{JK14a}.
\end{proof}

The question of what happens in case $Kv$ is separably closed  is addressed in the next section.

\section{Separably closed residue fields} \label{scf}
In this section, we give a partial answer to Question \ref{mainq} in case the residue field is separably closed. 
Recall that when $(K,v)$ is henselian and 
the residue field is not separably closed, we may add a symbol for 
$\mathcal{O}_v$ to \emph{any} NIP field structure on $K$ and obtain an NIP structure.
First, we note that
we cannot expect the same when it comes to separably closed (residue) fields:
\begin{Ex}[{\cite[Example 5.5]{HHJ19}}]
Let $K$ be a separably closed field and assume that $v_1$ and $v_2$ are 
two independent valuations on $K$. Then 
$(K,v_1,v_2)$ has IP in $\mathcal{L}_\mathrm{ring}\cup \{\mathcal{O}_1\} \cup \{\mathcal{O}_2\}$.
\end{Ex}

There are of course many examples of separably 
closed fields with independent valuations:
\begin{Ex} \label{extra} Let $\mathbb{Q}^\mathrm{alg}$ be an algebraic closure of $\mathbb{Q}$ 
and let $p\neq l$ be prime.
Consider a prolongation $v_p$ (respectively $v_l$) of the $p$-adic (respectively $l$-adic) valuation on $\mathbb{Q}$ to
$\mathbb{Q}^\mathrm{alg}$. Then $v_p$ and $v_l$ are independent, thus the bi-valued field $(\mathbb{Q}^\mathrm{alg},v_p,v_l)$
has IP.
\end{Ex}

As any separably closed valued field has NIP in  $\mathcal{L}_\mathrm{ring}\cup \{\mathcal{O}\}$ and any valuation is henselian
on a separably closed field, we 
cannot expect an analogue of Corollary \ref{nonsep} to hold for 
separably closed residue fields. We will instead
focus on the following version of Question \ref{mainq}:
\begin{Q} \label{vmainq}
Let $K$ be NIP as a pure field and $v$ a henselian valuation on $K$ with $Kv$ separably closed. 
Is $(K,v)$ NIP in $\mathcal{L}_\mathrm{ring} \cup \{\mathcal{O}\}$?
\end{Q}

\subsection{Ingredients for the proof of Theorem \ref{B}}
We will split the proof of Theorem \ref{B} into an equicharacteristic case and a mixed characteristic case. In both
cases, separably algebraically maximal Kaplansky fields play an important role.
\begin{Def} 
Let $(K,v)$ be a valued field and $p=\mathrm{char}(Kv)$. 
\begin{enumerate}
\item
We say that $(K,v)$ is \emph{(separably) algebraically maximal} if $(K,v)$ has no immediate (separable) algebraic extensions.
\item We say that $(K,v)$ is \emph{Kaplansky} if $p=0$ or if $p>0$ and the value group $vK$ is $p$-divisible and the residue field $Kv$ is perfect
and admits no Galois extensions of degree divisible by $p$.
\end{enumerate}
\end{Def}
Note that separable algebraic maximality always implies henselianity. See \cite{Kuh13} for more details on (separably)
algebraically maximal Kaplansky fields.
As mentioned in the introduction, there is a transfer theorem which works for separably algebraically maximal Kaplansky fields:
\begin{Thm}[{\cite[Theorem 3.3]{JS}} and {\cite[Proposition 4.1]{AJ19b}}] \label{SAMK}
Any complete theory of separably algebraically maximal Kaplansky fields 
is NIP if and only if corresponding theories of the residue field and value group are NIP.
\end{Thm}

The fact that the theory $\mathrm{SCVF}$ of separably closed valued fields is NIP has been proven (independently) by 
Delon and Hong; however, Delon's proof is unpublished and Hong's proof only works for finite degree of imperfection. It is also an immediate consequence of Theorem \ref{SAMK}:
\begin{Cor}[{\cite[Corollary 4.2]{AJ19b}}] \label{del}
Let K be separably closed and let $v$ be a valuation on $K$. 
Then $(K,v)$ has NIP as a pure valued field.
\end{Cor}
%\begin{proof}
%see Hong for finite degree of imperfection. Give Delon's proof for infinite degree.
%\end{proof}

Using an argument by Scanlon, we reduce Question \ref{mainq} to the case of algebraically closed residue fields.
\begin{Prop}[Scanlon] Let $(K,v)$ be a henselian valued field with $\mathrm{char}(Kv)=p$,
such that $Kv$ is not perfect and has no separable extensions of degree divisible by $p$. Then $\mathcal{O}_v$ is definable
in $\mathcal{L}_\mathrm{ring}$. \label{scan}
\end{Prop}
\begin{proof}
Choose $t \in \mathcal{O}_v$ such that we have $\bar{t} \in Kv \setminus Kv^p$. Consider the
$\mathcal{L}_\mathrm{ring}$-definable subset of $K$ given by
$$S:=\{a \in K \,|\, \exists\, L \supseteq K \textrm{ with }[L:K]<p \textrm{ and }\exists y \in L:\,y^p-ay=t\}.$$ 
We claim that $S = \{a \in K\,|\, v(a) \leq 0\}$ holds.
We first show the inclusion $S \subseteq \{a \in K\,|\, v(a) \leq 0\}$.
Assume for a contradiction that there is some $a \in S$ with $v(a)>0$. Take $L\supseteq K$ and $y \in L$ witnessing
$a \in S$, i.e., we have $[L:K]<p$ and $y^p-ay=t$. Let $w$ denote the unique prolongation of $v$ to $L$. Note that, as
$w(t) \geq 0$ and $w(a)>0$,
we have $w(y) \geq 0$. Hence, we get $\bar{y}^p=\bar{t} \in Lw$. However, as $[Lw:Kv]\leq[L:K]<p$, this gives the desired
contradiction.\\
For the other inclusion, suppose that we have $v(a)\leq 0$. Choose any $b \in K^\mathrm{alg}$ with
$b^{p-1}=a$ and set $L:=K(b)$. In particular, we have $[L:K]\leq p-1 <p$. Let $w$ denote the unique extension
of $v$ to $L$.
Consider the equation $$baZ^p-Zba-t= (bZ)^p-a(bZ)-t=0$$
over $L$.
As we have $w(ba)\leq 0$, this equation has a solution in $L$ if and only if the equation
$$Z^p - Z - \dfrac{t}{ba}=0$$
over $\mathcal{O}_w$ has a solution in $\mathcal{O}_w$. As $(L,w)$ is henselian and $Lw$, a separable extension of $Kv$, also 
has no separable extensions of degree divisible by $p$, there is some $z \in \mathcal{O}_w$
with $z^p-z=\frac{t}{ba}$. For $y = zb$, we conclude $y^p-ay=t$ as desired.\\
It now follows immediately from the claim that 
$\mathcal{O}_v$ is also definable.
\end{proof}

\subsection{Compositions of NIP valuations}
In the proof of Theorem \ref{main2}, we decompose the valuation $v$ on $K$ into several pieces: a
 coarsening $u$ of $v$ and a valuation
$\bar{v}$ on $Ku$ such that  $v$ is the composition of $\bar{v}$ and $u$. 
However, in general, it is not
clear whether showing that each of these is NIP is sufficient to show that 
$v$ is NIP. The situation is simpler if the residue field
$Ku$ of $u$ is stably embedded.
\begin{Def}
Let $M$ be a structure in some language $\mathcal{L}$ and $\mathcal{N} \succ M$
 sufficiently saturated. 
A definable set $D$ is said to be \emph{stably embedded} 
if for every formula $\phi(x;y)$, $y$ a finite tuple of variables from the same sort as $D$, there is a formula $d\phi(z;y)$ such that for any $a\in \mathcal{N}^{|x|}$, there is a tuple $b\in D^{|z|}$, such that $\phi(a;D)=d\phi(b;D)$.
\end{Def}

See \cite[Chapter 3]{Simon:book} for more on stable embeddedness.
Note that \cite[Proposition 2.5]{JS} proves that we can add NIP structure on a stably embedded set and stay NIP. This always works in separably algebraically maximal Kaplansky fields:

\begin{Prop}[{\cite[Lemma 3.1]{JS}} and {\cite[Proof of Proposition 4.1]{AJ19b}}] \label{SEK}
Let $(K,v)$ be a separably algebraically maximal Kaplansky field. Then, the residue field
$Kv$ is stably embedded.
\end{Prop}

There are more natural examples of henselian fields with stably embedded residue fields.

\begin{Def} Let $(K,v)$ be a valued field of characteristic $(\mathrm{char}(K), \mathrm{char}(Kv))=(0,p)$
for some prime $p >0$. We say that $(K,v)$ is 
\begin{enumerate}
\item
\emph{unramified} if $v(p)$ is the smallest positive element of $vK$ and
\item \emph{finitely ramified} if the interval $[0,v(p)] \subseteq vK$ is finite.
\end{enumerate}
\end{Def}

The residue field of any unramified henselian valued field is
purely stable embedded as an $\mathcal{L}_\textrm{ring}$-structure ({\cite[Corollary 13.7]{AJ19}}). 
However, the residue field of a finitely ramified henselian 
valued field need not be stably embedded as a pure
field (cf.~\cite[Example 12.8]{AJ19}), and it is not known whether it is stably embedded in case the residue
field is not perfect. Nonetheless, using that every finitely ramified henselian
valued field is up to elementary equivalence
a finite extension of an unramified field with the same residue field, one can nonetheless prove
an NIP transfer:

\begin{Prop}[{\cite[Corollary 4.7]{AJ19b}}] \label{fin}
Let $(K,v)$ be a henselian valued field of mixed characteristic and $u$ a coarsening of $v$ such that $(K,u)$ is finitely ramified and $(Ku, \bar{v})$ is NIP. Then $(K,v)$ is NIP.
\end{Prop}

We finish the section with some open problems.
\begin{Qs}
\begin{enumerate}
\item Is the composition of (henselian) NIP valuations NIP?
\item Is the residue field of every henselian valuation on an NIP field stably embedded?
\end{enumerate}
\end{Qs}
By \cite[Proposition 2.5]{JS}, a positive answer to the second question would imply
a positive answer to the henselian case of the first question. Moreover, there seems to be no known
example of a henselian field in $\mathcal{L}_\mathrm{ring} \cup \{\mathcal{O}_v\}$ such that the residue field is not stably embedded.

\subsection{The case of separably closed residue fields}
In this subsection, we prove our second main result which was mentioned as Theorem \ref{B} in the introduction.
We start with the equicharacteristic case:
\begin{Prop}
Let $K$ be NIP, $v$ henselian on $K$ with $\mathrm{char}(K)=\mathrm{char}(Kv)$.
Then, $(K,v)$ is NIP as a pure valued field. \label{Prop:equi} \label{equi}
\end{Prop}
\begin{proof} We may assume that $v$ is non-trivial as otherwise the statement is clear.
In case $Kv$ is non-separably closed, the statement follows from Corollary \ref{nonsep}. Now assume that $Kv$
is separably closed, in particular, $Kv$ is NIP as a pure field. Moreover, we 
assume that $K$ is not separably closed
since otherwise the conclusion follows from Corollary \ref{del}.
If $\mathrm{char}(Kv)=0$, the statement follows immediately from Delon's classical result (\cite[Theorem A.15]{Simon:book}) - or by the fact that 
any equicharacteristic $0$ henselian valued field
is separably algebraically maximal Kaplansky. 
On the other hand, if $\mathrm{char}(K)=p>0$, then $K$ admits no Galois extensions of degree divisible by $p$
by \cite[Corollary 4.4]{KSW}. Thus, $vK$ is $p$-divisible and 
$Kv$ is perfect (for an argument for the latter, see the proof of 
\cite[Proposition 4.1]{JS}). As $Kv$ is separably closed, 
we conclude that $(K,v)$ is Kaplansky. Moreover, any immediate
separable extension of $K$ has degree divisible by $p$ by the lemma of
Ostrowski 
(\cite[see (3) on p.~280 for the statement and p.~300 for the proof]{Kuh11}). 
Thus, $(K,v)$ is separably algebraically maximal with algebraically closed residue field.
By Theorem \ref{SAMK}, $(K,v)$ is NIP.
\end{proof}

We now come to the general case:
\begin{Thm} \label{main2}
Let $K$ be NIP, $v$ henselian on $K$.
Then $(K,v)$ is NIP as
a pure valued field. 
\end{Thm}
\begin{proof} 
If $Kv$ is not separably closed, the statement follows from \ref{nonsep}. In the
case when $Kv$ is separably closed and non-perfect, the theorem
holds by Proposition \ref{scan}. Thus, we may assume that $Kv$ is algebraically closed. Moreover, by Corollary \ref{del}, 
we may assume that $K$ is not
separably closed. The equicharacteristic case 
follows from Proposition \ref{Prop:equi}. Thus, we now assume 
$\mathrm{char}(K)=0$ and $\mathrm{char}(Kv)=p>0$.
Furthermore, we assume that $(K,v)$ is
$\aleph_1$-saturated. 

We consider the standard decomposition of $v$ (writing $\Gamma:=vK$):
Let $\Delta_p \leq \Gamma$ be the biggest convex subgroup not containing $v(p)$ and let $\Delta_0 \leq \Gamma$ be
the smallest convex subgroup containing $v(p)$. We get the following decomposition of the place $\varphi_v:K \to Kv$
corresponding to $v$:
$$K =K_0 \xlongrightarrow{\Gamma/\Delta_0} K_1 \xlongrightarrow{\Delta_0/\Delta_p} K_2  \xlongrightarrow{\Delta_p} K_3=Kv$$
where every arrow is labelled with the corresponding value group.
Note that $\mathrm{char}(K)=\mathrm{char}(K_1)=0$ and $\mathrm{char}(K_2)= \mathrm{char}(Kv)=p$.
Let $v_i$ denote the valuation on $K_i$ corresponding to the place $K_i \to K_{i+1}$.

By \cite[Theorem 1.13]{AnKuh16}, the value group $v_1K_1=\Delta_0/\Delta_p$ of $(K_1,v_1)$ is either isomorphic to $\mathbb{Z}$ or $\mathbb{R}$.
Moreover, by saturation (and since $\Delta/\Delta_0$ has rank $1$), $(K_1,v_1)$ is spherically complete and thus algebraically
maximal (compare
also the proof of \cite[Lemma 6.8]{Joh15}). 
We now consider two cases:

In case $v_1K_1$ is isomorphic to $\mathbb{Z}$, the composition $u$ of $v_1$ and 
$v_0$ is finitely ramified. 
Since $u$ is henselian, Corollary \ref{resnip} implies that $Ku=K_2$ is an NIP field, thus
$(K_2,v_2)$ is NIP (Proposition \ref{Prop:equi}).
Applying Proposition \ref{fin}, we
conclude that $(K,v)$ is NIP.
 
On the other hand, in case $(K_1,v_1)$ has divisible value group, we first show that 
$K_2$ is perfect.
Assume $K_2$ is not perfect. Recall that it is NIP by Corollary \ref{resnip}, and hence admits
no Galois extensions of degree divisible by $p$. Hence, by Proposition \ref{scan}, the composition $u$ of $v_1$ and
$v_2$ is again definable. But this contradicts $\aleph_1$-saturation: Recall that $\Delta_p$ is 
the biggest convex subgroup not containing $v(p)$, and that $\Delta_0$ is the smallest convex subgroup
containing $v(p)$. If $\Delta_p$ is definable in $(K,v)$, then there is always a minimum positive
element in $\Delta_0/\Delta_p$ since by saturation, $\Delta_0/\Delta_p$ must otherwise contain a convex subgroup.
Thus, if $\Delta_0/\Delta_p$ is isomorphic to $\mathbb{R}$, $K_2$ is perfect.

We now argue that $(K_1, v_1)$ is an algebraically maximal Kaplansky field. 
By what we have just shown, its
residue field $K_2$ is perfect and NIP (using Corollary \ref{resnip} again), and
by assumption we are in the case when
$(K_1,v_1)$ is divisible. Thus, $(K_1,v_1)$ is
Kaplansky. Since we have already argued that $(K_1,v_1)$
is algebraically maximal, Proposition \ref{SAMK} now implies that $(K_1,v_1)$ is NIP. 
By \ref{Prop:equi}, also $(K,v_0)$ and $(K_2,v_2)$ are NIP.
Moreover, by Proposition \ref{SEK}, $K_2 = K_1v_1$ is stably embedded as a pure field in $(K_1,v_1)$
and of course, being an equicharacteristic $0$ henselian valued field, $K_1=Kv_0$
is stably embedded as a pure field in $(K,v_0)$. Thus, applying \cite[Proposition 2.5]{JS}
twice, we finally conclude that $(K,v)$ is NIP.
\end{proof}

\section{Ordered fields} \label{ord}
In this section, we use the same technique as in the proof of Proposition \ref{real} 
to study convex
valuation rings on an ordered field. 
We show that any convex valuation ring $\mathcal{O}_v$ on $K$ is definable in $(K,<)^\mathrm{Sh}$. The
idea to consider convex valuation rings on ordered fields was suggested by Salma Kuhlmann.

\begin{Def} Let $(K,<)$ be an ordered field and $R \subseteq K$ a subring.
\begin{enumerate}
\item 
The \emph{$<$-convex hull of $R$ in $K$} is defined as
$$\mathcal{O}_R(<) :=\{x \in K\, :\,x,-x < a \textrm{ for some }a \in R\}.$$
\item We say that $R$ is \emph{$<$-convex} if  $\mathcal{O}_R(<)=R$.
\end{enumerate}
\end{Def}

The following facts about convex valuation rings are well-known.
\begin{Fact}[{\cite[p.\,36]{EP05}}] Let $(K,<)$ be an ordered valued field. 
\begin{enumerate}
\item Any convex subring of $K$ containing $1$ is a valuation ring.
\item A subring $R \subseteq K$ is $<$-convex if and only if $R$ is a convex subgroup of the additive group of $K$. 
Thus, any two valuations $v,w$ on $K$ which are convex with respect to $<$ are comparable.
\item There is a (unique) finest valuation $v_0$ on $K$ which is convex with respect to $<$. 
It is called
the \emph{natural valuation} of $(K,<)$.
The valuation ring $\mathcal{O}_{v_0}$ is the convex hull of the integers in $(K,<)$.
\end{enumerate}
\end{Fact}

It is now an easy consequence of the properties of the natural valuation that convex valuation rings are definable in the 
Shelah expansion:
\begin{Prop}
Let $(K,<)$ be an ordered field and $\mathcal{O}_v$ a convex valuation ring on $K$. Then
$\mathcal{O}_v$ on $K$ is definable in $(K,<)^\mathrm{Sh}$.
\end{Prop}
\begin{proof}
As the valuation ring of the natural valuation $v_0$ is exactly the convex closure of $\mathbb{Z}$ in $K$, it is definable in 
$(K,<)^\mathrm{Sh}$. As any convex valuation $v$ on $K$ is a coarsening of $v_0$, the valuation ring of $v$
is also definable in $(K,<)^\mathrm{Sh}$.
\end{proof}

Applying Proposition \ref{Shelah}, this yields the following
\begin{Cor} \label{cord}
Let $K$ be an ordered field such that $\mathrm{Th}(K)$ is NIP in some language 
$\mathcal{L}\supseteq \mathcal{L}_\mathrm{of}$ and
let $v$ be a convex valuation on $K$. Then, $(K,v)$ is NIP in $\mathcal{L}\cup\{\mathcal{O}_v\}$.
\end{Cor}

\section*{Acknowledgements} I would like to thank Tom Scanlon for his interest and sharing his proof of Proposition \ref{scan}
with me. Furthermore, I would like to thank Fran\c{c}oise Delon for explaining
to me her (unpublished) proof of Theorem \ref{del}, 
and Salma Kuhlmann for pointing out that the proof idea of Proposition 
\ref{real} also applies to convex valuation rings. My thanks furthermore goes
to Sylvy Anscombe, Martin Bays, Immanuel Halupczok and Pierre Simon 
for helpful discussions.
This research was funded by the Deutsche Forschungsgemeinschaft (DFG, German Research Foundation) via CRC 878 and under Germany's Excellence Strategy
EXC 2044-390685587, Mathematics M\"unster: Dynamics--Geometry--Struc\-ture
as well as by a Fellowship from the Daimler und Benz Stiftung.
\bibliographystyle{alpha}
\bibliography{franzi}
\end{document}